\documentclass[11pt]{article}

\usepackage{amsmath,amssymb,amsthm}
\usepackage{geometry}
\usepackage{hyperref}
\usepackage{tikz}
\usepackage{caption}
\usepackage{cancel}
\usepackage{enumitem}
\usepackage{mathabx}
\usepackage[backend=bibtex,style=numeric]{biblatex}
\hypersetup{
  colorlinks   = true,
  linkcolor    = blue,
  citecolor    = red,
  urlcolor     = blue
}

\newtheorem{theorem}{Theorem}[section]
\newtheorem{lemma}[theorem]{Lemma}
\newtheorem{proposition}[theorem]{Proposition}
\newtheorem{corollary}[theorem]{Corollary}
\newtheorem{conjecture}[theorem]{Conjecture}
\theoremstyle{definition}

\theoremstyle{remark}
\newtheorem{remark}[theorem]{Remark}

\newcommand{\Rm}{\mathrm{Rm}}
\newcommand{\Rc}{\mathrm{Rc}}

\newcommand{\diff}[1]{\, \mathrm{d}#1}
\newcommand{\Hess}{\mbox{Hess}}

\newcommand{\tr}{\mathrm{tr}_g}
\newcommand{\ts}[1]{\mathrm{#1}}
\title{Rigidity of Gradient Shrinking Ricci Solitons with a Vanishing Bach-like Tensor and Related Variational Formulas}
\author{James Siene\\\\Department of Mathematics, Lehigh University}
\date{}

\begin{document}
\maketitle
\begin{abstract}
The classical Bach tensor in four dimensions can be expressed as a linear combination of two independent, symmetric, divergence-free, quadratic-in-curvature tensors $\ts U$ and $\ts V$. Several classification results for gradient-shrinking Ricci solitons have been obtained under the assumption that the Bach tensor vanishes. We define a Bach-like tensor to be any other linear combination of $\ts U$ and $\ts V$. We prove that within a certain cone of parameters $\mathcal{C}$, the vanishing of a Bach-like tensor forces a four-dimensional complete gradient-shrinking Ricci soliton to be either Einstein or isometric to the Gaussian soliton, extending the results of Cao--Chen \cite{caochen:2013}. The special case where $\mathrm{U}=0$ forces $f_{\min}\in\{0,1,2\}$, with rigidity holding when $f_{\min}=0,2$. The remaining case $f_{\min}=1$ is the central open problem, with $\mathbb{R}^2\times \mathbb{S}^2(\sqrt 2)$ as the conjectured exceptional geometry. Finally, we show that Bach-like tensors arise as Euler--Lagrange equations of a two-parameter family of quadratic curvature functionals and compute the corresponding first and second variation formulas.
\end{abstract}

\begin{small}\tableofcontents\end{small}

\section{Introduction and Main Results}

\par A \emph{gradient Ricci soliton} is a Riemannian manifold \((M^n, g)\) together with a smooth potential function \(f: M \to \mathbb{R}\) satisfying
\begin{equation}\label{eq:soliton}
\ts{Rc} + \nabla^2 f \;=\; \rho\,g,
\end{equation}
for $\rho$ a constant. The soliton is said to be \emph{shrinking} if $\rho > 0$, and it is common to scale $f$ so that $\rho = \frac{1}{2}$ for convenience. We normalize $f$ throughout so that
\begin{equation}\label{eq:normalization}
    \ts R + |\nabla f|^2 = f,
\end{equation}
following Hamilton \cite{hamilton:1988}. Such solitons arise as self-similar solutions of Hamilton's Ricci flow, as models for type I singularities. Unless otherwise stated, our work will be in dimension $n=4$.

\par
The classical Bach tensor was introduced by R. Bach \cite{bach:1921} to study conformal relativity. It is defined in dimension 4 by:
\begin{equation}\label{eq.bachtensor}
    \ts B_{ij} = \nabla^k\nabla^l \ts W_{ijkl} + \frac{1}{2} \ts R^{kl}\ts W_{ikjl},
\end{equation}
where $\ts W_{ijkl}$ denotes the Weyl curvature and $\ts R_{ij}$ the Ricci curvature. It is well-known that $\ts B_{ij}$ is conformally invariant, symmetric, trace-free, divergence-free, and quadratic in the Riemann curvature tensor. As seen in \cite{bergman:2004} there are exactly two independent symmetric $2$-tensors that are divergence-free and quadratic in curvature in four dimensions:
\begin{equation}\label{eq.utensor}
    \ts U_{ij} = 2\ts R_{ipjq}\ts R^{pq} + \Delta \ts R_{ij} - \frac{1}{2}|\ts{Rc}|^2 g_{ij} - \ts R\, \ts R_{ij} - \frac{1}{2}\Delta \ts R \, g_{ij} + \frac{1}{4}\ts R^2 \, g_{ij},
\end{equation}
\begin{equation}\label{eq.vtensor}
    \ts V_{ij} = -\nabla_i \nabla_j \ts R + \Delta \ts R \, g_{ij} + \ts R\, \ts R_{ij} - \frac{1}{4}\ts R^2\, g_{ij}.
\end{equation}
One can express the Bach tensor as a combination of $\ts U$ and $\ts V$. In dimension 4:
\begin{equation}\label{eq.bachliketensor}
    \ts B_{ij} = \frac{1}{2} \ts U_{ij} + \frac{1}{6}\ts V_{ij}.
\end{equation}
\par We define a \emph{Bach-like tensor} to be of the form:
\begin{equation}
    \mathfrak{B}_{ij} = \alpha \,\ts U_{ij} + \beta \,\ts V_{ij}
\end{equation}
for constants $(\alpha,\beta)\in\mathbb{R}^2\setminus\{\mathbf{0}\}$. If $\alpha/\beta=3$ then the Bach-like tensor is simply a rescaled Bach tensor. H.-D.\ Cao and Q.\ Chen showed in \cite{caochen:2013} that if the Bach tensor vanishes on a complete gradient-shrinking Ricci soliton, then it must be rigid. Once we move beyond the Bach tensor itself, we can show that Bach-like flatness still implies rigidity within a certain parameter cone, but without the index-heavy computations typical of the Bach-flat case. The principal results of this paper are as follows.

\begin{theorem}[\textbf{Vanishing V-tensor}]\label{theorem.vanishingV}
    Let $(M^4,g,f)$ be a complete gradient-shrinking Ricci soliton satisfying \eqref{eq:soliton} and \eqref{eq:normalization}. Then $\ts V_{ij} \equiv 0$ if and only if $M^4$ is either Einstein or isometric to the Gaussian shrinking soliton. In particular, $f_{\min}\in\{0,2\}$ correspondingly.
\end{theorem}

\begin{theorem}[\textbf{Vanishing U-tensor}]\label{theorem.vanishingU}
    Let $(M^4,g,f)$ be a complete gradient-shrinking Ricci soliton satisfying \eqref{eq:soliton} and \eqref{eq:normalization}. If $\ts U_{ij} \equiv 0$ then $f_{\min}\in\{0,1,2\}$. If $f_{\min}=0$ then $M^4$ is isometric to the Gaussian shrinking soliton, and if $f_{\min}=2$ then $M^4$ is Einstein. Conversely, every Einstein or Gaussian shrinking soliton satisfies $\ts U_{ij}\equiv 0$.
\end{theorem}

The case $f_{\min}=1$ in Theorem~\ref{theorem.vanishingU} is not resolved by the methods of this paper. At the minimum point $p$ of $f$, the tensor $\ts U$ vanishes independently of any global assumption when $\ts R(p)=1$, so the local analysis yields no information. By contrast, $\ts V$ cannot vanish at $p$ when $f_{\min}=1$. We state the following as a conjecture.

\begin{conjecture}\label{conj.cylinder}
    Let $(M^4,g,f)$ be a complete gradient-shrinking Ricci soliton satisfying \eqref{eq:soliton} and \eqref{eq:normalization}. If $\ts U\equiv 0$ and $f_{\min}=1$, then $M^4$ is isometric to $\mathbb{R}^2\times S^2(\sqrt{2})$.
\end{conjecture}

The candidate space $\mathbb{R}^2\times S^2(\sqrt{2})$ is verified to satisfy the hypotheses. For the shrinking soliton normalization $\rho=\tfrac{1}{2}$, the soliton equation in the $S^2$ directions forces radius $r=\sqrt{2}$, giving $\ts R=1$ and $f_{\min}=1$ via \eqref{eq:normalization}. Indexing by $0,1$ in the $\mathbb{R}^2$-directions and $a,b$ on $S^2(\sqrt{2})$, a direct computation gives $\ts U_{00} = \ts U_{0a} = \ts U_{ab} = 0$, so $\ts U\equiv 0$. Moreover, $\ts V_{00} = -\tfrac{1}{4}$ and $\ts V_{ab} = \tfrac{1}{4}g_{ab}$, so $\ts V\not\equiv 0$, confirming the space is neither Einstein nor Gaussian. This computation was suggested to us by H.-D.\ Cao.

\begin{theorem}[\textbf{Bach-like Rigidity}]\label{theorem.vanishingBachLike}
    Let $(M^4,g,f)$ be a complete gradient-shrinking Ricci soliton satisfying \eqref{eq:soliton} and \eqref{eq:normalization}. Assume $(\alpha,\beta)\in\mathcal{C}$ where
    \[\mathcal{C} = \{(\alpha,\beta)\,\,|\,\, \alpha\geq 0 \,\,\mathrm{and}\,\, \beta>\alpha/3\}\cup\{(\alpha,\beta)\,\,|\,\, \alpha\leq 0\,\,\mathrm{and}\,\, \beta<\alpha/3\},\]
    and assume $3\beta\neq\alpha$. If $\mathfrak{B}_{ij} = \alpha\,\ts U_{ij} + \beta\, \ts V_{ij} = 0$, then $M^4$ is either Einstein or isometric to the Gaussian shrinking soliton. Conversely, every Einstein or Gaussian shrinking soliton satisfies $\mathfrak{B}_{ij}=0$. If $3\beta=\alpha$, the rigidity results of Cao--Chen \cite{caochen:2013} apply.
\end{theorem}

We approach the proof of Theorem~\ref{theorem.vanishingBachLike} by primarily using weighted integral identities from the $\ts V_{ij}$ tensor inside the cone $\mathcal{C}$. This result covers the entire $(\alpha, \beta)$ plane aside from the Bach line, the $\beta=0$ axis, and the origin; a complete treatment of the remaining cases is contingent on Conjecture~\ref{conj.cylinder} and is deferred to future work.

\section{Integral Identities}
To proceed, we rely on a number of integral identities implied by $\Delta \ts R = 0$. The connection between harmonic scalar curvature and Bach-like tensors is rooted in the \emph{traces} of $\ts U$ and $\ts V$.

\begin{lemma}\label{lemma.tracesUV}
    Let $\ts U_{ij}$ and $\ts V_{ij}$ be given by \eqref{eq.utensor} and \eqref{eq.vtensor} respectively. Then
    \begin{align*}
        \mathrm{tr}(\ts U_{ij}) &= -\Delta \ts R,\\
        \mathrm{tr}(\ts V_{ij}) &= 3\Delta \ts R.
    \end{align*}
\end{lemma}
So if $(\alpha,\beta)\in\mathcal{C}$, then $\mathfrak{B}_{ij} = 0$ implies $\Delta \ts R = 0.$ We now derive the integral identities for both the unweighted and weighted volume elements.
\begin{lemma}[Integral Identities] \label{lemma.integralIdentities}
        Let $(M^4, g, f)$ be a complete gradient-shrinking Ricci soliton satisfying \eqref{eq:soliton} and \eqref{eq:normalization}, and assume that $\Delta \ts R =0$. Let $r$ be a regular value of $f$ and $\Omega_r = \{ x \in M \, | \, f(x) \leq r \}$. Then:
        \begin{enumerate}
            \item[1.)] $\displaystyle \int_{\partial \Omega_r} \frac{1}{|\nabla f|}\langle \nabla \ts R, \nabla f \rangle \diff{S} = 0,$
            
            \item[2.)] $\displaystyle \int_{\Omega_r}\langle \nabla \ts R, \nabla f \rangle \diff{V} = 0,$

            \item[3.)] $\displaystyle \int_{\Omega_r} |\nabla \ts R|^2 \diff{V} = \int_{\partial \Omega_r} \frac{\ts R}{|\nabla f|}\langle \nabla \ts R, \nabla f \rangle \diff{S} = -\int_{\partial \Omega_r}|\nabla f| \langle \nabla \ts R, \nabla f \rangle \diff{S},$
            
            \item[4.)] $\displaystyle \int_{\Omega_r} \langle \nabla \ts R, \nabla f\rangle \, \ts e^{-f} \diff{V} = 0,$
            
            \item[5.)] $\displaystyle \int_{\Omega_r} \ts R\, \ts{Rc}(\nabla f, \nabla f)\, \ts e^{-f} \diff{V} = \frac{1}{2}\int_{\partial \Omega_r} |\nabla f| \langle \nabla \ts R, \nabla f \rangle \ts e^{-f} \diff{S} + \frac{1}{2}\int_{\Omega_r} |\nabla \ts R|^2 \, \ts e^{-f} \diff{V},$
            
            \item[6.)] $\displaystyle \int_{\Omega_r} f\langle \nabla \ts R, \nabla f \rangle \, \ts e^{-f} \diff{V} = 0.$
        \end{enumerate}
    \end{lemma}
    \begin{proof}
        First, we have:
        \begin{equation*}
            0 = \int_{\Omega_r} \Delta \ts R \diff{V} = \int_{\partial \Omega_r} \frac{1}{|\nabla f|} \langle \nabla \ts R, \nabla f \rangle \diff{S}
        \end{equation*}
        showing (1). Now we integrate by parts to get:
        \begin{align*}
            \int_{\Omega_r} \langle \nabla \ts R, \nabla f \rangle \diff{V} &= \int_{\partial \Omega_r} \frac{f}{|\nabla f|} \langle \nabla \ts R, \nabla f \rangle \diff{S} - \int_{\Omega_r} f \Delta \ts R \diff{V}
            \\
            &= r\int_{\partial \Omega_r}\frac{1}{|\nabla f|}\langle \nabla \ts R, \nabla f \rangle \diff{S} - \int_{\Omega_r} f \Delta \ts R \diff{V}.
        \end{align*}
        As $\Delta \ts R = 0$ by assumption, and using (1) we see that both integrals must vanish, showing (2). To show (3) we integrate $\ts R\Delta \ts R$:
        \begin{equation*}
            0 = \int_{\Omega_r} \ts R\Delta \ts R \diff{V} = \int_{\partial \Omega_r} \frac{\ts R}{|\nabla f|} \langle \nabla \ts R, \nabla f \rangle \diff{S} - \int_{\Omega_r} |\nabla \ts R|^2 \diff{V}.
        \end{equation*}
        And by using the normalization \eqref{eq:normalization}, we get (3). Now we have:
        \begin{align*}
            0 &= \int_{\Omega_r} \Delta \ts R \, \ts e^{-f} \diff{V}
            \\
            &= \int_{\partial \Omega_r} \frac{\ts e^{-f}}{|\nabla f|}\langle \nabla \ts R, \nabla f \rangle \diff{S} + \int_{\Omega_r} \langle \nabla \ts R, \nabla f \rangle \, \ts e^{-f} \diff{V}
            \\
            &= \ts e^{-r}\int_{\partial \Omega_r} \frac{1}{|\nabla f|}\langle \nabla \ts R, \nabla f \rangle \diff{S} + \int_{\Omega_r} \langle \nabla \ts R, \nabla f \rangle \, \ts e^{-f} \diff{V}
            \\
            &= \int_{\Omega_r} \langle \nabla \ts R, \nabla f \rangle \, \ts e^{-f} \diff{V}
        \end{align*}
        showing (4). To show (5) we integrate $\ts R\Delta \ts R$ with the weighted volume $\ts e^{-f} \diff{V}$:
        \begin{align*}
            0 &= \int_{\Omega_r} \ts R\Delta \ts R \, \ts e^{-f} \diff{V}
            \\
            &= \int_{\partial\Omega_r} \frac{\ts R}{|\nabla f|} \langle \nabla \ts R, \nabla f \rangle \, \ts e^{-f} \diff{S} - \int_{\Omega_r} \langle \nabla \big(\ts R\, \ts e^{-f}\big), \nabla \ts R \rangle \diff{V}
            \\
            &= \int_{\partial\Omega_r} \frac{\ts R}{|\nabla f|} \langle \nabla \ts R, \nabla f \rangle \, \ts e^{-f} \diff{S} - \int_{\Omega_r} \Big( |\nabla \ts R|^2 - \ts R \langle \nabla \ts R, \nabla f \rangle \Big) \, \ts e^{-f} \diff{V}
            \\
            &= -\int_{\partial\Omega_r} |\nabla f| \langle \nabla \ts R, \nabla f \rangle \, \ts e^{-f} \diff{S} - \int_{\Omega_r} \Big( |\nabla \ts R|^2 - 2\ts R\, \ts{Rc}(\nabla f, \nabla f) \Big) \, \ts e^{-f} \diff{V}.
        \end{align*}
        And lastly for (6), we have:
        \begin{align*}
            \int_{\Omega_r} f\langle \nabla \ts R, \nabla f \rangle \, \ts e^{-f} \diff{V} &= - \int_{\Omega_r} \langle f \nabla \,\ts R, \nabla \ts e^{-f} \rangle \diff{V}
            \\
            &= -\int_{\partial \Omega_r} \frac{f}{|\nabla f|}\langle \nabla \ts R, \nabla f \rangle \, \ts e^{-f} \diff{S} + \int_{\Omega_r} \mathrm{div}(f\, \nabla \ts R) \ts e^{-f} \diff{V}
            \\
            &= -r\ts e^{-r}\int_{\partial \Omega_r} \frac{1}{|\nabla f|}\langle \nabla \ts R, \nabla f \rangle  \diff{S} + \int_{\Omega_r} \mathrm{div}(f \,\nabla \ts R) \ts e^{-f} \diff{V}
            \\
            &= \int_{\Omega_r} \Big( \langle \nabla \ts R, \nabla f \rangle + f \,\Delta \ts R\Big) \ts e^{-f} \diff{V}
            \\
            &= 0.
        \end{align*}
    \end{proof}

\section{The Vanishing V-tensor}
    The tensor $\ts V_{ij}$ has fewer and less complicated terms than $\ts U_{ij}$. We will obtain integral identities by assuming $\ts V_{ij} = 0$, and use these to prove Theorem~\ref{theorem.vanishingV}. We first consider the Hessian term.

    \begin{lemma}\label{lemma.hessianRintegral}
        Let $(M^4, g, f)$ be a complete gradient-shrinking Ricci soliton satisfying \eqref{eq:soliton} and \eqref{eq:normalization}, and assume that $\Delta \ts R = 0$. Then if $r$ is a regular value of $f$ and $\Omega_r =  \{ x \in M \, | \, f(x) \leq r \}$, we have:
        \begin{equation*}
            \int_{\Omega_r} \nabla^2 \ts R(\nabla f, \nabla f) \, \ts e^{-f} \diff{V} = -\ts e^{-r} \int_{\Omega_r} |\nabla \ts R|^2  \diff{V} + \frac{1}{2}\int_{\Omega_r} |\nabla \ts R|^2 \, \ts e^{-f} \diff{V}.
        \end{equation*}
    \end{lemma}
    \begin{proof}
        We integrate:
        \begin{align*}
            \int_{\Omega_r} \nabla^2 &\ts R(\nabla f, \nabla f) \, \ts e^{-f} \diff{V}
            \\ &= \int_{\Omega_r} \langle \nabla \langle \nabla \ts R, \nabla f \rangle, \nabla f \rangle \, \ts e^{-f} \diff{V} - \frac{1}{2}\int_{\Omega_r} \langle \nabla \ts R, \nabla f \rangle \, \ts e^{-f} \diff{V}\\
            &\hspace{0.5in}+ \frac{1}{2} \int_{\Omega_r} |\nabla \ts R|^2 \, \ts e^{-f} \diff{V}
            \\
            &= -\int_{\Omega_r} \langle \nabla \langle \nabla \ts R, \nabla f \rangle, \nabla \ts e^{-f} \rangle \diff{V} + \frac{1}{2} \int_{\Omega_r} |\nabla \ts R|^2 \, \ts e^{-f} \diff{V}
            \\
            &= \int_{\partial \Omega_r} |\nabla f| \langle \nabla \ts R, \nabla f \rangle \ts e^{-f} \diff{S} + \int_{\Omega_r} \Delta (\ts e^{-f}) \langle \nabla  \ts R, \nabla f \rangle \diff{V}\\
            &\hspace{0.5in}+ \frac{1}{2}\int_{\Omega_r} |\nabla \ts R|^2 \, \ts e^{-f} \diff{V}
            \\
            &= \int_{\partial \Omega_r} |\nabla f| \langle \nabla \ts R, \nabla f \rangle \ts e^{-f} \diff{S} + \int_{\Omega_r} (|\nabla f|^2 - \Delta f) \langle \nabla \ts R, \nabla f \rangle \ts e^{-f} \diff{V}\\
            &\hspace{0.5in}+ \frac{1}{2}\int_{\Omega_r} |\nabla \ts R|^2 \, \ts e^{-f} \diff{V}
            \\
            &= \int_{\partial \Omega_r} |\nabla f| \langle \nabla \ts R, \nabla f \rangle \ts e^{-f} \diff{S} + \int_{\Omega_r} (f - 2)\langle \nabla \ts R, \nabla f \rangle \ts e^{-f}\diff{V}\\
            &\hspace{0.5in}+ \frac{1}{2} \int_{\Omega_r} |\nabla \ts R|^2 \, \ts e^{-f}\diff{V}.
        \end{align*}
        Lemma~\ref{lemma.alteredIntegralIdentities}(1) and (3) shows that the integral of $(f-2)\langle \nabla \ts R,\nabla f\rangle$ vanishes, giving us our desired result.
    \end{proof}

    It will be necessary to control the $|\nabla \ts R|^2$ integral as we allow $r\to \infty$. Work already exists for controlling the integral of $|\ts{Rc}|^2$, which we are able to translate to our problem.

    \begin{lemma}\label{lemma.gradRbounded}
        Let $(M^n,g_{ij},f)$ be a complete gradient-shrinking Ricci soliton satisfying \eqref{eq:soliton}. Then:
        \begin{equation*}
            \int_M |\nabla \ts R|^2 \ts e^{-\alpha f}\diff{V} < \infty
        \end{equation*}
        for any $\alpha > 0$.
    \end{lemma}
    \begin{proof}
        By the Cauchy-Schwarz inequality we have:
        \begin{equation*}
            \frac{1}{4}|\nabla \ts R|^2 = |\mathrm{Rc}(\nabla f)|^2 \leq |\mathrm{Rc}|^2 |\nabla f|^2.
        \end{equation*}
        By H.-D.\ Cao and D.\ Zhou's growth estimates for the potential function \cite{caozhou:2010}, we have:
        \begin{equation*}
            |\nabla f|^2(x) \leq \frac{1}{4}(d(x) + c_2)^2,
        \end{equation*}
        where $d(x)$ is the distance function from a fixed point $x_0$ and $c_1,c_2$ are constants depending on $n$ and the geometry of $g_{ij}$ of the unit ball $B_{x_0}(1)$. Hence, there exists a compact set $K$ where $|\nabla f|^2 \ts e^{\frac{-\alpha f}{2}} \leq 1$ on $M\backslash K$. We have
        \begin{equation*}
            \int_M |\mathrm{Rc}|^2 |\nabla f|^2 \ts e^{-\alpha f} \diff{V}= \int_K |\mathrm{Rc}|^2 |\nabla f|^2 \ts e^{-\alpha f} \diff{V} +  \int_{M \backslash K} |\mathrm{Rc}|^2 |\nabla f|^2 \ts e^{-\alpha f} \diff{V}.
        \end{equation*}
        The first integral over the compact set $K$ is finite. We show that the second integral is also finite:
        \begin{equation*}
            \begin{split}
                \int_{M \backslash K} |\ts{Rc}|^2 |\nabla f|^2 \ts e^{-\alpha f} \diff{V} &=  \int_{M \backslash K} (|\ts{Rc}|^2 \ts e^{-\alpha f/2}) |\nabla f|^2 \ts e^{-\alpha f/2}\diff{V}\\
                &\leq  \int_{M \backslash K} |\ts{Rc}|^2 \ts e^{-\alpha f/2}\diff{V}.
            \end{split}
        \end{equation*}
        By O.\ Munteanu and N.\ Sesum's results in \cite{munteanusesum:2013}, we have:
        \begin{equation*}
            \int_M |\ts{Rc}|^2 \ts e^{-\lambda f} \diff{V} < \infty,
        \end{equation*}
        for any $\lambda > 0$, so after setting $\lambda = \frac{\alpha}{2}$ the result follows.
    \end{proof}

    \begin{lemma}\label{lemma.gradRvanishes}
        Let $(M^n,g_{ij},f)$ be a complete gradient-shrinking Ricci soliton satisfying \eqref{eq:soliton}. Then if $\Omega_r = \{ f\leq r\}$ for regular values $r > 0$ of $f$ and $\alpha >0$:
        \begin{equation*}
            \lim_{r\rightarrow \infty} \Big(\ts e^{-\alpha r}\int_{\Omega_r} |\nabla \ts R|^2 \diff{V}\Big) = 0.
        \end{equation*}
    \end{lemma}
    \begin{proof}
    We first apply Lemma~\ref{lemma.gradRbounded}:
    \begin{equation*}
        \int_M |\nabla \ts R|^2 \ts e^{-\alpha f/2} \diff{V} < \infty.
    \end{equation*}
    Since $f \leq r$ on $\Omega_r$, we also have $\ts e^{-\alpha r/2} \leq \ts e^{-\alpha f/2}$:
    \begin{equation*}
        \frac{\ts e^{-\alpha r}}{2}\int_{\Omega_r} |\nabla \ts R|^2 \diff{V} = \frac{\ts e^{-\alpha r/2}}{2} \int_{\Omega_r} |\nabla \ts R|^2 \ts e^{-\alpha r/2}\diff{V} \leq \frac{\ts e^{-\alpha r/2}}{2}\int_{\Omega_r} |\nabla \ts R|^2 \ts e^{-\alpha f/2}\diff{V}.
    \end{equation*}
    Now let $r\rightarrow \infty$, giving us the desired result:
    \begin{equation*}
        \lim_{r\rightarrow \infty}\Big(\frac{\ts e^{-\alpha r}}{2}\int_{\Omega_r}|\nabla \ts R|^2 \diff{V}\Big) \leq \lim_{r\rightarrow \infty}\Big(\frac{\ts e^{-\alpha r/2}}{2}\int_{\Omega_r} |\nabla \ts R|^2 \ts e^{-\alpha f/2}\diff{V}\Big) = 0.
    \end{equation*}
    \end{proof}

    \section{Proof of Theorem~\ref{theorem.vanishingV}}
    \begin{proof}
        Let $\Omega_r = \{ x \in M \, | \, f(x) \leq r \}$ where $r$ is a regular value of $f$. By Lemma~\ref{lemma.tracesUV}, $\ts V\equiv 0$ implies $\Delta \ts R = 0$. By Lemma~\ref{lemma.integralIdentities}(5) and (3) we have:
        \begin{align*}
             \int_{\Omega_r} \ts R\, \ts{Rc}(\nabla f, \nabla f) \, \ts e^{-f} \diff{V} &= -\frac{1}{2}\int_{\Omega_r} \ts R\Delta_f \ts R \ts e^{-f}\diff{V}
            \\
            &= -\frac{1}{2}\int_{ \Omega_r} |\nabla \ts R|^2 \, \ts e^{-r} \diff{V} + \frac{1}{2}\int_{\Omega_r} |\nabla \ts R|^2 \, \ts e^{-f} \diff{V}.
        \end{align*}
        And using Lemma~\ref{lemma.hessianRintegral}, we compute the integral of $\ts V(\nabla f, \nabla f)$ with respect to the weighted volume element $\ts e^{-f}\diff{V}$, noting that the term involving $\Delta \ts R$ vanishes:
        \begin{align*}
            \int_{\Omega_r} \ts V(&\nabla f, \nabla f) \, \ts e^{-f} \diff{V} \\
            &= -\int_{\Omega_r} \Big(\nabla^2 \ts R(\nabla f, \nabla f) - \ts R \, \mathrm{Rc}(\nabla f, \nabla f) + \frac{1}{4}\ts R^2 |\nabla f|^2 \Big)\, \ts e^{-f} \diff{V}
            \\
            &= \frac{\ts e^{-r}}{2}\int_{\Omega_r} |\nabla \ts R|^2 \diff{V} - \frac{1}{4}\int_{\Omega_r}\ts R^2|\nabla f|^2 \, \ts e^{-f} \diff{V}.
        \end{align*}
        Letting $r\rightarrow \infty$, by Lemma~\ref{lemma.gradRvanishes} the first integral vanishes. We obtain:
        \begin{equation}
            \label{eq.VweightedIntegral}
            \int_M \ts V(\nabla f, \nabla f)\, \ts e^{-f}\diff{V} = -\int_{M}\frac{1}{4}\ts R^2 |\nabla f|^2 \, \ts e^{-f} \diff{V}.
        \end{equation}
        Assuming $\ts V_{ij} = 0$, we must have either $\ts R = 0$ or $|\nabla f| = 0$ pointwise. By the work of S.\ Pigola, M.\ Rimoldi, and A.\ Setti \cite{pigolarimoldi:2011}, either $\ts R>0$ everywhere, hence $\nabla f = 0$ and the manifold is Einstein, or if $\ts R=0$ anywhere, then the soliton must be isometric to the Gaussian soliton on $\mathbb{R}^4$. The values $f_{\min}\in\{0,2\}$ follow from \eqref{eq:normalization}: if Einstein then $\nabla f\equiv 0$ and $\mathrm{Rc}=\tfrac{1}{2}g$, giving $\ts R=2$ and $f_{\min}=2$; if Gaussian then $\ts R=0$ and $f_{\min}=0$.

        For the converse, suppose $(M^4, g, f)$ is Einstein. Then as the scalar curvature is constant, we have:
        \begin{align*}
            \ts V_{ij} &= \ts R\, \ts R_{ij} - \frac{1}{4} \ts R^2 g_{ij}
            = -\ts R \Big(\ts R_{ij} - \frac{1}{4} \ts R \, g_{ij}\Big) = 0.
        \end{align*}
        If the manifold is Gaussian, then $\ts R=0$ and again $\ts V=0$.
    \end{proof}

    \section{Weighted Bach Integral}
    In \cite{caochen:2013} H.-D.\ Cao and Q.\ Chen consider how the Bach tensor links to the Cotton tensor of a conformal metric, called $\ts D_{ijk}$. The norm of $\ts D_{ijk}$ contains rich geometric information, and they were able to use it to prove their classification results. We will link Bach-like tensors to $\ts D_{ijk}$ as well. Recall that the Cotton tensor is given by:
    \[
        \ts C_{ijk} = -\frac{n-2}{n-3}\nabla_l \ts W_{ijkl},
    \]
    where $\ts W_{ijkl}$ is the Weyl curvature. The Cotton tensor is skew-symmetric in the first two indices and trace-free in any two indices:
    \[
    \ts C_{ijk} = - \ts C_{jik} \quad \mathrm{and} \quad g^{ij}\,\ts C_{ijk} = g^{ik}\,\ts C_{ijk} = 0.
    \]
    The Cotton tensor for the conformal metric $\displaystyle \widehat{g}_{ij}= \ts e^{\frac12 f}\, g_{ij}$ is given by:
    \[
        \ts D_{ijk} = \ts C_{ijk} + \ts W_{ijkl}\,\nabla_l f.
    \]
    On a complete gradient-shrinking Ricci soliton, $\ts D_{ijk}$ can be expressed as:
    \begin{align*}
        \ts D_{ijk}&= \frac{1}{n-2}\big(\ts R_{jk}\nabla_i f - \ts R_{ik} \nabla_j f\big) + \frac{1}{2(n-1)(n-2)}\big(g_{jk}\,\nabla_i \ts R - g_{ik}\,\nabla_j \ts R\big)\\
        &\quad -\frac{1}{(n-1)(n-2)} \ts R \, \big(g_{jk}\, \nabla_i f - g_{ik} \, \nabla_j f\big).
    \end{align*}
    H.-D.\ Cao and Q.\ Chen \cite{caochen:2013} showed how the Bach tensor can be expressed in terms of $\ts C_{ijk}$ and $\ts D_{ijk}$:
    \[
        \ts B_{ij} = -\frac{1}{n-2}\Big(\nabla_k \ts D_{ikj} + \frac{n-3}{n-2} \ts C_{jli}\nabla_l f\Big).
    \]

    \begin{lemma}\label{lemma.weightedBachIntegral}
        Let $(M^n,g_{ij},f)$ be a complete gradient-shrinking Ricci soliton satisfying \eqref{eq:soliton}. Then
        \begin{equation*}
            \int_M \ts B(\nabla f, \nabla f) e^{-f} \diff{V} = -\frac{1}{2}\int_M |\ts D_{ijk}|^2 e^{-f} \diff{V}.
        \end{equation*}
    \end{lemma}
    \begin{proof}
        Let $r$ be a regular value of $f$ and $\Omega_r=\{f\leq r\}$. We use the antisymmetric properties of the Cotton and $\ts D_{ijk}$ tensor:
        \begin{align*}
            (n-2)\int_{\Omega_r} &\ts B_{ij} \nabla_i f\nabla_j f \,\ts e^{-f}\diff{V}
            \\
            &= -\int_{\Omega_r} \nabla_k \ts D_{ikj} \nabla_i f\nabla_j f \ts e^{-f}\diff{V} - \frac{n-3}{n-2} \int_{\Omega_r} \ts C_{jki} \nabla_i f \nabla_j f \nabla_k  f \,\ts e^{-f}\diff{V}\\
            &= - \int_{\partial \Omega_r} \ts D_{ikj} \nabla_i f\nabla_j f\nabla_k f \cdot |\nabla f|^{-1}\, \ts e^{-f} \diff{S}  \\
            &\quad + \int_{\Omega_r} \ts D_{ikj}(\ts R_{ik}\nabla_j f + \ts R_{kj}\nabla_i f) \, \ts e^{-f}\diff{V}\\
            &\quad- \int_M \ts D_{ikj} \nabla_i f\nabla_j f\nabla_k f \, \ts e^{-f}\diff{V}
            \\
            &= \int_{\Omega_r} \ts D_{ikj} (\ts R_{ik} \nabla_j f + \ts R_{kj}\nabla_i f)\, \ts e^{-f} \diff{V}\\
            &= -\frac{1}{2}\int_{\Omega_r} \ts D_{ijk} ( \ts R_{ik}\nabla_j f - \ts R_{kj} \nabla_i f) \, \ts e^{-f} \diff{V}\\
            &= - \frac{(n-2)}{2}\int_{\Omega_r} |\ts D_{ijk}|^2\, \ts e^{-f} \diff{V}.
        \end{align*}
    \end{proof}

    \section{Proof of Theorem~\ref{theorem.vanishingBachLike}}
        First observe that for any $(\alpha,\beta)\in \mathcal{C}$, we can write:
        \[
            \mathfrak{B}_{ij} = 2\alpha \, \ts B_{ij} + \Big(\beta - \frac{\alpha}{3}\Big) \, \ts V_{ij}.
        \]
        Assuming $\mathfrak{B}_{ij} = 0$ and taking a trace we have:
        \[
            g^{ij}\,\mathfrak{B}_{ij} = (3\beta-\alpha)\Delta \ts R = 0 \implies \Delta \ts R = 0,
        \]
        by Lemma~\ref{lemma.tracesUV}. Then we have:
        \begin{align*}
            \int_M \mathfrak{B}(\nabla f, \nabla f)\, \ts e^{-f}\diff{V} & = -\alpha \int_M |\ts D_{ijk}|^2 \, \ts e^{-f} \diff{V} - \frac14\Big(\beta-\frac{\alpha}{3}\Big)\int_M \ts R^2\,|\nabla f|^2 \, \ts e^{-f}\,\diff{V} = 0.
        \end{align*}
        Since $(\alpha,\beta)\in\mathcal{C}$, if $\alpha > 0$ and $\beta > \alpha/3$ then both integrals are nonnegative and must each vanish, so $|\ts D_{ijk}|^2 \equiv 0$ and $\ts R^2|\nabla f|^2 \equiv 0$. We conclude that $M$ is either Einstein or isometric to the Gaussian soliton as in the proof of Theorem~\ref{theorem.vanishingV}. The argument for $\alpha < 0$ and $\beta < \alpha/3$ is analogous.

    \section{The Vanishing U-tensor}
    The analysis we used for the vanishing $V$-tensor doesn't completely carry over to $\mathrm{U}=0$, but will allow us to prove Theorem \ref{theorem.vanishingU}. Define a weighted volume element of the form
    \[
    \diff{V}_c(\mathbf{x}) = \ts{e}^{-cf(\textbf{x})}\diff{V} \quad \mathrm{where} \quad c>0.
    \]
    \begin{lemma}\label{lemma.alteredIntegralIdentities}
        Let $(M^4,f,g)$ be a complete gradient-shrinking Ricci soliton satisfying \eqref{eq:soliton} and \eqref{eq:normalization}. Let $r$ be a regular value of $f$ and define the sub-level set $\Omega_r = \{ x \in M \, | \, f(x) \leq r \}$. If $c>0$ is a constant and $\Delta \ts R = 0$, then we have the following:
        \begin{enumerate}
            \item[(1.)] $\displaystyle \int_{\Omega_r}\langle\nabla \ts R, \nabla f\rangle \diff{V}_c = 0$,
            \item[(2.)] $\displaystyle\int_{\Omega_r} \ts R\langle\nabla \ts R, \nabla f\rangle \diff{V}_c = \frac{1}{c}\int_{\Omega_r} |\nabla \ts R|^2 \diff{V}_c - \frac{\ts e^{-c\, r}}{c}\int_{\Omega_r}|\nabla \ts R|^2 \diff{V}$,
            \item[(3.)] $\displaystyle\int_{\Omega_r}f\,\langle\nabla \ts R, \nabla f\rangle \diff{V}_c = 0$.
        \end{enumerate}
    \end{lemma}
    \begin{proof}
        First we have:
        \begin{align*}
            0 &= \int_{\Omega_r} \Delta \ts R \ts e^{-cf}\diff{V}\\
            &= \int_{\partial \Omega_r} \frac{1}{|\nabla f|}\langle \nabla \ts R, \nabla f\rangle \ts e^{-cf} \diff{S} + c\int_{\Omega_r} \langle \nabla \ts R, \nabla f\rangle \ts e^{-cf}\,\diff{V}\\
            &= c\int_{\Omega_r} \langle\nabla \ts R,\nabla f\rangle \ts e^{-cf}\, \diff{V},
        \end{align*}
        showing (1). Next, we have:
        \begin{align*}
            0 &= \int_{\Omega_r} \ts R \Delta \ts R \,\ts e^{-cf}\diff{V}\\
            &= \int_{\partial \Omega_r} \frac{\ts R}{|\nabla f|}\langle \nabla \ts R, \nabla f \rangle \, \ts e^{-cf} \diff{S} - \int_{\Omega_r} \Big(|\nabla \ts R|^2 - c\, \ts R\, \langle \nabla \ts R, \nabla f \rangle \Big) \ts e^{-cf}\diff{V}.
        \end{align*}
        Hence, using Lemma~\ref{lemma.integralIdentities}(3),
        \begin{equation*}
            \int_{\Omega_r} \ts R\,\langle\nabla \ts R, \nabla f\rangle \, \ts e^{-cf}\diff{V} = \frac{1}{c}\int_{\Omega_r} |\nabla \ts R|^2 \ts e^{-cf}\diff{V} - \frac{\ts e^{-c\, r}}{c}\int_{\Omega_r}|\nabla \ts R|^2 \diff{V},
        \end{equation*}
        showing (2). And the last identity is:
        \begin{align*}
            \int_{\Omega_r} f&\langle\nabla \ts R, \nabla f\rangle \ts e^{-cf}\diff{V}\\
            &= -\frac{1}{c}\int_{\Omega_r} \langle f\, \nabla \ts R, \nabla (\ts e^{-cf})\rangle \diff{V}\\
            &= -\frac{1}{c}\Big[ \int_{\partial\Omega_r} \frac{f}{|\nabla f|} \langle\nabla \ts R, \nabla f\rangle \ts e^{-cf}\diff{S} - \int_{\Omega_r}\mathrm{div}(f\nabla \ts R)\ts e^{-cf}\diff{V} \Big]\\
            &= -\frac{1}{c}\Big[r \, \ts e^{-cr}\int_{\partial\Omega_r}\frac{1}{|\nabla f|} \langle\nabla \ts R, \nabla f\rangle \ts e^{-cf}\diff{S} - \int_{\Omega_r}\mathrm{div}(f\,\nabla \ts R)\ts e^{-cf}\diff{V} \Big]\\
            &= \frac{1}{c}\int_{\Omega_r} \Big(\langle \nabla \ts R, \nabla f\rangle + f\Delta \ts R\Big)\ts e^{-cf}\diff{V}\\
            &=0,
        \end{align*}
        showing (3).
    \end{proof}

    We first state a quick consequence of Bochner's formula.
    \begin{lemma}\label{lemma.riccibochner}
        Let $(M^4,f,g)$ be a complete gradient-shrinking Ricci soliton satisfying \eqref{eq:soliton} and \eqref{eq:normalization}. If $\Delta \ts R = 0$, then:
        \begin{equation*}
            2|\Rc|^2 = \ts R + \langle \nabla \ts R, \nabla f\rangle.
        \end{equation*}
    \end{lemma}
    \begin{proof}
        By Bochner's formula:
        \begin{align*}
            \frac{1}{2}\Delta |\nabla f|^2 &= \langle \nabla (\Delta f), \nabla f\rangle + |\Hess(f)|^2 + \Rc(\nabla f, \nabla f)\\
            \frac{1}{2}\Delta (f - \ts R) &= \langle \nabla(2 - \ts R),\nabla f\rangle + |\Hess(f)|^2 + \frac{1}{2}\langle \nabla \ts R, \nabla f \rangle\\
            \frac{1}{2}(2-\ts R)&= -\frac{1}{2}\langle \nabla \ts R, \nabla f\rangle + \left|\frac{1}{2}g - \Rc\right|^2\\
            1-\frac{\ts R}{2} &= -\frac{1}{2}\langle \nabla \ts R, \nabla f\rangle + 1 - \ts R + |\Rc|^2\\
            \frac{\ts R}{2} &= -\frac{1}{2}\langle \nabla \ts R, \nabla f\rangle + |\Rc|^2.
        \end{align*}
        From here our claim clearly follows.
    \end{proof}

    \begin{lemma}\label{lemma.uIntegral}
        Let $(M^4,g,f)$ be a complete gradient-shrinking Ricci soliton satisfying \eqref{eq:soliton} and \eqref{eq:normalization}. Assume that $\Delta \ts R \equiv 0$, $r > 0$ is a regular value of $f$, and $\Omega_r = \{f\leq r\}$. Then
        \[
            \int_{\Omega_r} \ts U(\nabla f, \nabla f)\diff{V}_c =\frac14 \int_{\Omega_r}\big(\ts R^2 - 2\,|\Rc|^2\big) \, |\nabla f|^2 \diff{V}_c.
        \]
        Moreover, this can be rewritten as
        \[
            \int_{\Omega_r} \ts U(\nabla f, \nabla f)\diff{V}_c = \frac14 \int_{\Omega_r} \big( \ts R^2\,|\nabla f|^2 - \ts R \, |\nabla f|^2 + \tfrac1c |\nabla \ts R|^2\big)\diff{V}_c - \frac{\ts e^{-cr}}{4c}\int_{\Omega_r} |\nabla \ts R|^2 \diff{V}.
        \]
    \end{lemma}
    \begin{proof}
        One can easily check that on a soliton we have
        \begin{equation*}
            \ts U_{ij} = \ts R_{ij} + \nabla_k \ts R_{ij} \nabla_k f - \frac{1}{2}|\mathrm{Rc}|^2 g_{ij} - \ts R\, \ts R_{ij} - \frac{1}{2}\Delta \ts R \, g_{ij} + \frac{\ts R^2}{4}g_{ij},
        \end{equation*}
        as well as the soliton identities
        \begin{enumerate}
            \item[(a)] $\mathrm{Rc}(\nabla f) = \frac{1}{2}\nabla \ts R \implies \mathrm{Rc}(\nabla f, \nabla f) = \frac{1}{2}\langle \nabla \ts R, \nabla f \rangle$,
            \item[(b)] $\Delta f = 2-\ts R$,
            \item[(c)] $\nabla^2f(\nabla f) = \frac{1}{2}\nabla f - \frac{1}{2}\nabla \ts R$.
        \end{enumerate}
        The second term in our expression for $\ts U$ is simplified by looking at the following divergence:
        \begin{align*}
            \nabla_k (\ts R_{ij}& \nabla_i f\nabla_j f \nabla_k f \, \ts e^{-cf})\\
            &= \nabla_k \ts R_{ij} \nabla_i f\nabla_j f\nabla_k f \ts e^{-cf} + 2\, \ts R_{ij}\nabla_i f\nabla_j \nabla_k f (\nabla_k f) \ts e^{-cf} \\
            &\quad+ \Delta f \, \ts R_{ij} \nabla_i f\nabla_j f \ts e^{-cf}- c|\nabla f|^2 \ts R_{ij} \nabla_i f \nabla_j f\, \ts e^{-cf}.
        \end{align*}
        Now by the divergence theorem we have
        \begin{align*}
            \int_{\Omega_r} &\nabla_k \ts R_{ij}\nabla_i f\nabla_j f\nabla_k f \, \ts e^{-cf}\diff{V}\\
            & = \int_{\Omega_r}\mathrm{div}\left(\frac{1}{2}\langle \nabla \ts R,\nabla f\rangle \nabla_k f \, \ts e^{-cf}\right)\diff{V} + \frac{1}{2} \int_{\Omega_r} \ts R\langle \nabla \ts R, \nabla f\rangle \, \ts e^{-cf}\diff{V} \\
            &\quad- \frac{1}{2}\int_{\Omega_r} |\nabla \ts R|^2 \, \ts e^{-cf}\diff{V}+ \frac{\ts e^{-cr}}{2}\int_{\Omega_r} |\nabla \ts R|^2 \diff{V} \\
            &\quad- \int_{\Omega_r}\langle \nabla \ts R, \frac{1}{2}\nabla f - \frac{1}{2}\nabla \ts R\rangle \, \ts e^{-cf}\diff{V}\\
            & = \frac{1}{2}\int_{\partial \Omega_r} |\nabla f|\langle \nabla \ts R, \nabla f \rangle \, \ts e^{-cf}\diff{S} + \frac{1}{2c}\int_{\Omega_r}|\nabla \ts R|^2 \, \ts e^{-cf}\diff{V}\\
            &\quad - \frac{\ts e^{-cr}}{2c}\int_{\Omega_r}|\nabla \ts R|^2 \diff{V} + \frac{\ts e^{-cr}}{2}\int_{\Omega_r}|\nabla \ts R|^2 \diff{V}\\
            &=\int_{\Omega_r} \ts R\, \mathrm{Rc} (\nabla f, \nabla f) \diff{V}_c.
        \end{align*}
        So the weighted integral of $\ts U_{ij}$ becomes
        \begin{equation*}
            \int_{\Omega_r} \ts U(\nabla f, \nabla f)\diff{V}_c = \frac{1}{4}\int_{\Omega_r} (\ts R^2-2|\Rc|^2)|\nabla f|^2 \diff{V}_c.
        \end{equation*}
        Now transforming using the Bochner formula from Lemma~\ref{lemma.riccibochner}:
        \begin{align*}
            \int_{\Omega_r} &\ts U(\nabla f, \nabla f)\diff{V}_c\\
            &= \frac14\int_{\Omega_r} \Big( \ts R^2|\nabla f|^2 - \ts R|\nabla f|^2 + \ts R\langle \nabla \ts R, \nabla f\rangle \Big) \diff{V}_c\\
            &= \frac{1}{4}\int_{\Omega_r} \Big( \ts R^2|\nabla f|^2 - \ts R|\nabla f|^2 + \frac{1}{c}|\nabla \ts R|^2\Big) \diff{V}_c - \frac{\ts e^{-cr}}{4c}\int_{\Omega_r}|\nabla \ts R|^2\diff{V}.
        \end{align*}
    \end{proof}

    The issue is that $\ts R^2 - 2\,|\Rc|^2$ can change sign over $M$, so
    \[
        \int_M \big(\ts R^2-2\,|\Rc|^2\big)\,|\nabla f|^2\diff{V}_c = 0
    \]
    does not yield obvious rigidity. For comparison, we state the analogous identity for $\ts V(\nabla f, \nabla f)$ under $\diff{V}_c$. The contrast with Lemma~\ref{lemma.uIntegral} is instructive: for $c\geq 1$ the integrand for $\ts V$ acquires a nonnegative term, while the integrand for $\ts U$ remains sign-indefinite for any $c>0$. This further highlights the subtleties in the $\ts U =0$ case.

    \begin{lemma}\label{lemma.alteredvintegral}
        Let $(M^4,f,g)$ be a complete gradient-shrinking Ricci soliton satisfying \eqref{eq:soliton} and \eqref{eq:normalization}. Let $r>0$ be a regular value of $f$ and $\Omega_r = \{f\leq r\}$. If $\Delta \ts R = 0$ and $c>0$, then
        \[
            \int_{\Omega_r} \ts V(\nabla f, \nabla f)\diff{V}_c = -\int_{\Omega_r}\Big(\frac{1}{4}\ts R^2|\nabla f|^2 + \frac{3(c-1)}{2c}|\nabla \ts R|^2\Big)\diff{V}_c + \frac{(4c-3)\ts e^{-cr}}{2c}\int_{\Omega_r}|\nabla \ts R|^2 \diff{V},
        \]
        and in particular
        \begin{equation*}
            \int_M \ts V(\nabla f,\nabla f)\diff{V}_c = -\int_{M}\Big(\frac14\ts R^2|\nabla f|^2 + \frac{3(c-1)}{2c}|\nabla \ts R|^2\Big)\diff{V}_c.
        \end{equation*}
    \end{lemma}

    \begin{proof}
        We compute the following divergence:
        \begin{align*}
            \nabla_i &(\nabla_j \ts R \nabla_i f\nabla_j f\,\ts e^{-cf}) \\
            &= \nabla_i\nabla_j \ts R \nabla_i f\nabla_j f \ts e^{-cf} + \Delta f \langle\nabla \ts R, \nabla f\rangle \ts e^{-cf}\\
            &\quad+\nabla^2f(\nabla f, \nabla \ts R)\ts e^{-cf} - c|\nabla f|^2\langle\nabla \ts R,\nabla f\rangle \ts e^{-cf}.
        \end{align*}
        Hence the integral of $\nabla^2 \ts R(\nabla f, \nabla f)\,\ts e^{-cf}$ is:
        \begin{align*}
            \int_{\Omega_r} & \nabla^2 \ts R(\nabla f,\nabla f)\,\ts e^{-cf} \diff{V}\\
            &= \int_{\Omega_r} \nabla_i(\langle\nabla \ts R,\nabla f\rangle \nabla_i f \,\ts e^{-cf})\diff{V} + \int_{\Omega_r}(\ts R-2)\langle\nabla \ts R, \nabla f\rangle\, \ts e^{-cf} \diff{V}\\
            &\quad -\int_{\Omega_r} \Big( \frac{1}{2}\langle\nabla \ts R,\nabla f\rangle - \frac{1}{2}|\nabla \ts R|^2\Big) \ts e^{-cf}\diff{V} + c\int_{\Omega_r}(f-\ts R)\langle\nabla \ts R, \nabla f\rangle \ts e^{-cf}\diff{V}\\
            &=\int_{\partial\Omega_r}|\nabla f|\langle\nabla \ts R, \nabla f\rangle \ts e^{-cf}\diff{S}\\
            &\quad+ (c-1)\int_{\Omega_r} \ts R\langle\nabla \ts R,\nabla f\rangle \ts e^{-cf}\diff{V} + \frac{1}{2}\int_{\Omega_r}|\nabla \ts R|^2 \ts e^{-cf}\diff{V}\\
            &= -\ts e^{-cr}\Big(\frac{c-1}{c}+1\Big)\int_{\Omega_r}|\nabla \ts R|^2 \diff{V} + \Big(\frac{c-1}{c}+\frac{1}{2}\Big)\int_{\Omega_r}|\nabla \ts R|^2 \ts e^{-cf} \diff{V}.
        \end{align*}
        Therefore:
        \begin{align*}
            \int_{\Omega_r} &\ts V(\nabla f,\nabla f)\, \ts e^{-cf}\diff{V}\\
            &= \int_{\Omega_r}\nabla^2 \ts R(\nabla f,\nabla f)\, \ts e^{-cf}\diff{V} +\frac{1}{2}\int_{\Omega_r} \ts R\langle\nabla \ts R,\nabla f\rangle \, \ts e^{-cf}\diff{V} - \frac{1}{4}\int_{\Omega_r}\ts R^2|\nabla f|^2\, \ts e^{-cf}\diff{V}\\
            &= -\int_{\Omega_r}\Big(\frac{1}{4}\ts R^2|\nabla f|^2 + \frac{3(c-1)}{2c}|\nabla \ts R|^2\Big)\, \ts e^{-cf}\diff{V} + \frac{(4c-3)\ts e^{-cr}}{2c}\int_{\Omega_r}|\nabla \ts R|^2 \diff{V}.
        \end{align*}
        Taking the limit $r\rightarrow \infty$ gives the stated global identity.
    \end{proof}

    \section{Proof of Theorem~\ref{theorem.vanishingU}}
    \begin{proof}
    Assume that $\ts U_{ij} = 0$. By Lemma~\ref{lemma.tracesUV}, $\Delta \ts R = 0$. Let $p$ be the point where $f$ attains its minimum value $f_{\min}$. Then $\nabla f(p) = 0$, so by \eqref{eq:normalization}, $\ts R(p) = f_{\min}$. Using $\ts U = 0$ at $p$ with $\Delta \ts R = 0$:
    \[
        \Rc(p)\big(1-\ts R(p)\big) + \tfrac14\ts R(p)\big(\ts R(p)-1\big)g(p) = 0,
    \]
    which factors as
    \[
        \big(1 - \ts R(p)\big)\!\left(\Rc(p) - \tfrac14 \ts R(p)\,g(p)\right) = 0.
    \]
    If $\ts R(p)\neq 1$ then $\Rc(p) = \tfrac14 \ts R(p)g(p)$, which gives
    \[
        |\Rc(p)|^2 = \tfrac{1}{4}\ts R^2(p).
    \]
    Since $\nabla f(p) = 0$, Lemma~\ref{lemma.riccibochner} gives $2|\Rc(p)|^2 = \ts R(p)$. Substituting:
    \[
        \tfrac{1}{2}\ts R^2(p) = \ts R(p) \implies \ts R(p)\big(\ts R(p)-2\big) = 0.
    \]
    Hence $\ts R(p)\in\{0,2\}$. Together with the $\ts R(p)=1$ case, we have $f_{\min} = \ts R(p)\in\{0,1,2\}$.

    \textit{Case $f_{\min}=0$.} We have $\ts R(p)=0$, and the manifold is isometric to the Gaussian soliton by \cite{pigolarimoldi:2011}.

    \textit{Case $f_{\min}=2$.} Using the soliton equation,
    \[
        \ts R^2 - 2|\Rc|^2 = \ts R^2-2\left| \tfrac12 g - \nabla^2 f \right|^2 = \ts R^2-2\ts R + 2 - 2|\nabla^2 f|^2.
    \]
    At the minimum $p$ with $f_{\min}=2$, we have $\Delta f(p) = 2-\ts R(p)=0$. Since $p$ is a minimum, $\nabla^2 f(p)\geq 0$, and $\Delta f(p)=0$ forces $\nabla^2 f(p)=0$. Thus:
    \[
        \ts R^2(p)-2\ts R(p)+2-2|\nabla^2 f(p)|^2 = 4-4+2-0 = 2 > 0.
    \]
    By continuity, there exist $r>2$ and $\epsilon>0$ such that $\ts R^2-2\ts R+2-2|\nabla^2 f|^2 \geq \epsilon > 0$ on $\Omega_r$. Therefore, under $\ts U \equiv 0$,
    \[
        \int_{\Omega_r} (\ts R^2 - 2|\Rc|^2)|\nabla f|^2 \, \ts e^{-f} \diff{V} = 0
    \]
    implies $|\nabla f|^2 \equiv 0$ on $\Omega_r$, hence $\ts R = f - |\nabla f|^2 = 2$ on $\Omega_r$. By unique continuation for harmonic functions, $\ts R=2$ on all of $M$, and the soliton is Einstein by \cite{pigolarimoldi:2011}.

    For the converse, if $M$ is Einstein then $\ts B\equiv 0$ (as $\ts B$ vanishes on Einstein manifolds \cite{caochen:2013}) and $\ts V\equiv 0$ by Theorem~\ref{theorem.vanishingV}, so $\ts U = 2\ts B - \tfrac13\ts V \equiv 0$. If $M$ is Gaussian then all curvature vanishes and $\ts U\equiv 0$.

    \textit{$^*$Case $f_{\min}=1$.} At a minimum point $p$ the $\ts U$-tensor vanishes independently of any global assumption, and $\mathrm{V}$ \emph{cannot} vanish. Coupled with the sign issue for the $\mathrm{U}$-vanishing integral identity, this case remains open; see Conjecture~\ref{conj.cylinder}.
    \end{proof}

    \section{Variational Origins of U and V}
    The classical Bach tensor arises naturally in a number of ways. One approach is to start with the conformal Einstein equations and then study integrability conditions via divergence operations. A necessary condition to be conformally Einstein is for the Bach tensor $\mathrm{B}_{ij}$ to vanish. The approach taken by R.\ Bach \cite{bach:1921} was to consider the first variation of the quadratic Riemannian functional
    \[
        \mathrm{S}_W = \int |\mathrm{W}|^2\,\diff{V}_g.
    \]
    As seen in A.\ Besse \cite{besse:1987}, one has
    \[
        \mathrm{B} \,\,\propto\,\,  \mathrm{grad}\left(\int |\mathrm{W}|^2 \diff{V}_g\right),
    \]
    so the Bach tensor is the gradient of a Riemannian functional whose energy depends on how much a metric deviates from being conformally flat. In this section, we identify Bach-like tensors as the Euler--Lagrange gradients of natural quadratic curvature functionals. We work in the standard transverse-traceless (TT) gauge for the second variation, imposing for a variation $h$:
    \[
        \mathrm{(TT)\,gauge:}\quad\mathrm{Div}_g(h) = 0 \quad \mathrm{and} \quad \tr(h) = 0.
    \]
    Recall the Lichnerowicz Laplacian acting on symmetric $2$-tensors:
    \[
        (\Delta_L \, h)_{ij} = (\nabla^* \nabla \, h)_{ij} + 2\,\ts R_{ikjl}\,h_{kl} - \ts R_{ik} \, h_{kj} - \ts R_{jk} \, h_{ki},
    \]
    where $\nabla^*\nabla \, h = -\nabla^k\nabla_k \, h$.

    \begin{lemma}\label{lem:var-origin}
    In dimension $4$, define
    \[
    \mathcal{F}_{\alpha,\beta}(g)
    := -\,\alpha \int_M |\mathrm{W}_g|^2 \, \mathrm{d}V_g
    \;+\;\frac{\alpha/3 - \beta}{2}\int_M \mathrm{R}_g^2 \, \mathrm{d}V_g.
    \]
    Then the $L^2$-gradient satisfies $\mathrm{grad}\,\mathcal{F}_{\alpha,\beta} = \alpha\, \ts U + \beta\, \ts V$. In particular, the Euler--Lagrange equation for $\mathcal{F}_{\alpha,\beta}$ is $\mathfrak{B}_{ij} = \alpha\, \ts U_{ij}+\beta \, \ts V_{ij} =0$.
    \end{lemma}

    \begin{proof}
    In dimension four:
    \[
        \delta \!\int |\ts W|^2 \diff{V}_g= -2\!\int \langle \ts B,\delta g\rangle\diff{V}_g \quad \mathrm{and} \quad \delta\!\int \ts R^2\diff{V}_g = -2\!\int \langle \ts V,\delta g\rangle\diff{V}_g.
    \]
    Using $\ts B=\tfrac12 \,\ts U+\tfrac16\,\ts V$:
    \[
    \delta \mathcal{F}_{\alpha,\beta}
    = \int \Big\langle \alpha\!\left(\tfrac12 \ts U+\tfrac16 \ts V\right)
    \;-\; \tfrac{\alpha/3-\beta}{2}\cdot 2 \ts V,\, \delta g \Big\rangle \diff{V}_g
    = \int \langle \alpha \ts U+\beta \ts V,\, \delta g\rangle \diff{V}_g.
    \qedhere
    \]
    \end{proof}

    \begin{corollary}\label{cor:U-V-variations}
    \[
    \ts V \;=\; -\tfrac12\,\mathrm{grad}\!\left(\int_M \ts R^2\, \diff{V}_g\right),
    \qquad
    \ts U \;=\; -\,\mathrm{grad}\!\left(\int_M \big(|\ts W|^2-\tfrac16 \ts R^2\big)\, \diff{V}_g\right).
    \]
    \end{corollary}

    \begin{proposition}[Second variation of $\int \ts R^2$ in TT gauge]\label{prop:second-var-R2}
    Let $(M^4,g)$ be closed. For a TT variation $g_t=g+th+o(t)$, the second variation of $F(g):=\int_M \ts R^2\, \diff{V}_g$ is
    \[
    \delta^2 F[h,h]
    = \int_M \Big[
    -\,\ts R\,|\nabla h|^2
    -2\ts R\,\ts R_{ikjl}\,h_{ij}h_{kl}
    +2\ts R\,\ts R_{ik}\,h_{ij}h_{kj}
    +2\langle \Rc,h\rangle^2
    +\Big(\frac12\ts R^2-2\,\Delta \ts R\Big)|h|^2
    \Big]\, \diff{V}.
    \]
    \end{proposition}

    \begin{proof}
    Write the first variation as $\delta F=-2\int \langle \ts V,h\rangle\diff{V}_g$. Differentiating once more and using TT gauge:
    \[
    \delta^2F[h,h]=-2\int_M \langle \ts V'(h),h\rangle \diff{V}_g.
    \]
    Expanding using the standard TT-gauge variations
    \begin{align*}
    \delta \ts R&=-\langle \Rc,h\rangle,\\
    \delta(\Delta \ts R)&=\Delta(\delta \ts R)-\langle h,\nabla^2 \ts R\rangle,\\
    \delta \Rc&=\tfrac12(\nabla^*\nabla h)_{ij}+(\Rm\!*h)_{ij}-(\Rc\circ h + h\circ \Rc)_{ij},
    \end{align*}
    and integrating by parts yields the result.
    \end{proof}

    \begin{remark}
    If $(M^4, g)$ is Einstein, the main contributions to the second variation come from the derivative term $-\ts R\, |\nabla h|^2$ and the curvature terms involving $\ts R_{ikjl}\, h_{ij}\,h_{kl}$.
    \end{remark}

    \begin{proposition}[Second variation at Einstein backgrounds]\label{prop:second-var-full}
    On an Einstein $4$-manifold $(M,g_E)$ with $\Rc=\lambda g_E$ and $\ts R=4\lambda$,
    the second variation of $\mathcal{F}_{\alpha,\beta}(g)$ in TT gauge is
    \[
    \delta^2\mathcal{F}_{\alpha,\beta}[h,h]
    = \int_M \!\Big\langle
    \big(
    \alpha\,\Delta_L^2
    + \frac12 \ts R\,(\beta-\tfrac{4}{3}\alpha)\,\Delta_L
    + \ts R^2(\tfrac{5}{36}\alpha - \tfrac{1}{4}\beta)
    \big)h,\; h
    \Big\rangle \diff{V}.
    \]
    Equivalently, for each TT eigenmode $\Delta_L h_k=\mu_k h_k$,
    \[
    \delta^2\mathcal{F}_{\alpha,\beta}[h,h]
    = \sum_k p_{\alpha,\beta}(\mu_k)\,\|h_k\|_{L^2}^2,
    \quad
    p_{\alpha,\beta}(\mu)
    = \alpha\mu^2
    + \frac12 \ts R\,(\beta-\tfrac{4}{3}\alpha)\mu
    + \ts R^2(\tfrac{5}{36}\alpha - \tfrac{1}{4}\beta).
    \]
    \end{proposition}

    \begin{proof}
    In TT gauge at an Einstein metric:
    \[
    \ts B'(h)=\tfrac12\big(\Delta_L-\tfrac{\ts R}{3}\big)\big(\Delta_L-\tfrac{\ts R}{6}\big)h,\qquad
    \ts V'(h)=\frac12 \ts R\,\Delta_L h.
    \]
    The claimed polynomial follows by expanding
    \[
    \delta^2\mathcal{F}_{\alpha,\beta}[h,h]
    = -2\alpha \!\int_M \langle \ts B'(h),h\rangle \diff{V}_g
    \;-\; (\alpha/3-\beta)\!\int_M \langle \ts V'(h),h\rangle \diff{V}_g. \qedhere
    \]
    \end{proof}

    \begin{corollary}[Bach line and factorization]\label{cor:bach-line}
    On the Bach line $\beta=\alpha/3$:
    \[
    \delta^2\mathcal{F}_{\alpha,\beta}[h,h]
    =\alpha \int_M \left\langle
    \big(\Delta_L-\tfrac{\ts R}{3}\big)\big(\Delta_L-\tfrac{\ts R}{6}\big)h,\; h
    \right\rangle \diff{V}_g.
    \]
    \end{corollary}

    \begin{corollary}[Local rigidity by spectral positivity]\label{cor:local-rigidity}
    Let $(M^4,g_E)$ be Einstein with TT-spectrum of $\Delta_L$ contained in $[\mu_0,\infty)$. If $p_{\alpha,\beta}(\mu)>0$ for every TT-eigenvalue $\mu$, then $g_E$ is a strict local minimizer of $\mathcal{F}_{\alpha,\beta}$ modulo diffeomorphisms and scaling. In particular, there exists a neighborhood of $g_E$ in which $\alpha\, \ts U+\beta\, \ts V=0$ has no nontrivial solutions other than metrics isometric to $g_E$.
    \end{corollary}

    \begin{remark}[Flat/Gaussian background]
    If $\ts R=0=\ts W$, then $\ts B'(h)=\tfrac12\Delta_L^2 h$ and $\ts V'(h)\equiv 0$ on TT. Hence
    \[
    \delta^2\mathcal{F}_{\alpha,\beta}[h,h]=\alpha \int_M |\Delta h|^2\,\diff{V},
    \]
    which is strictly positive for $\alpha>0$ on compactly supported variations.
    \end{remark}

\printbibliography
\end{document}